\documentclass[12pt]{article}
\usepackage{epsfig}
\usepackage{amsfonts}
\usepackage{amsthm}
\usepackage{amsmath}
\usepackage{amssymb}
\usepackage{graphicx}
\usepackage{ifpdf}
\ifpdf
\fi

\allowdisplaybreaks[1]
\newtheorem{theorem}{Theorem}
\newtheorem{lemma}[theorem]{Lemma}
\newtheorem{proposition}[theorem]{Proposition}
\newtheorem{corollary}[theorem]{Corollary}

\def\Pr#1{{\mathbf P\bigl[{#1}\bigr]}}
\def\cond{{ \; \big\vert \; }}

\def\Sum{\displaystyle\sum}

\def\bar{\overline}

\def\D{{\mathcal D}}

\def\Nat{{\mathbb N_0}}
\def\C{{\mathcal C}}
\def\T{{\mathcal T}}

\def\tecka{{\;\mbox{.}}}
\def\carka{{\;\mbox{,}}}

\def\softl{{l\kern-0.3ex\raise0.1ex\hbox{'}\kern-0.10ex}}

\def\iuuk{{Computer Science Institute}}

\def\mffuk{{Faculty of Mathematics and Physics, Charles University}}
\def\addrms{{Ma\-lo\-stran\-sk\'e n\'am\-\v es\-t\'\i{} 25, 118 00 Prague 1, Czech Republic}}

\def\Ggacrnum{{GACR 201/09/0197}}
\def\Ggauknum{{GAUK 601812}}
\def\Giti{{Institute for Theoretical Computer Science of Charles University (ITI), which was supported as project 1M0545 by Czech Ministry of Education.}}

\def\cutprob{{0.88672}}
\def\cutsize{{1.33008}}
\def\cutcover{{1.127752}}
\def\cutred{{0.491979}}
\def\cutblue{{0.508021}}

\def\cutgirth{{637\,789}}
\def\cutzyka{{1.28571}}

\def\cutwormald{{1.32595}}

\def\cutupper{{0.9351}}
\def\cutuppern{{1.4026}}

\def\cutprgPred{{2^{-11}}}
\def\cutprgPblue{{1}}
\def\cutprgPrb{{2^{-17}}}
\def\cutprgPinit{{2^{-18}}}
\def\cutprgKone{{34\,919}}
\def\cutprgKtwo{{283\,974}}
\def\cutprgK{{318\,894}}
\def\cutprgTHOLD{{10^{-7}}}

\def\cutprgprec{{657400}}
\def\cutprgerrfla{{2^{-657400} \times 10^{197862}}}
\def\cutprgerr{{10^{-35}}}

\def\cutzdebup{{1.386}}
\def\cutzdeblow{{1.382}}

\def\cutindep{{1.3056}}

\def\randsize{{10^7}}
\def\randprob{{0.88696}}

\begin{document}

\title{Maximum edge-cuts in~cubic graphs with large girth and in~random cubic graphs}

\date{}

\author{
Franti\v sek Kardo\v s
\,\thanks{\,Institute of Mathematics, Faculty of Science, University of Pavol Jozef \v Saf\'arik,
Jesenn\'a 5, 041 54 Ko\v sice, Slovakia. 
This author was supported by Slovak Research and Development Agency under the contract no. APVV-0023-10. E-mail: {\tt frantisek.kardos@upjs.sk}.}
\and Daniel Kr\'a\softl
\,\thanks{\,\iuuk, \mffuk, \addrms. This work was done when the author was affiliated with {\Giti} This author was also supported by the grant \Ggacrnum.
E-mail: {\tt kral@iuuk.mff.cuni.cz}.}
\and Jan Volec
\,\thanks{\,\iuuk, \mffuk, \addrms. This author was supported by the grants {\Ggacrnum} and {\Ggauknum}. E-mail: {\tt volec@iuuk.mff.cuni.cz}.}
}

\maketitle

\begin{abstract}
We show that for every cubic graph $G$ with sufficiently large girth there exists a probability
distribution on edge-cuts in $G$ such that each edge is in a randomly chosen cut with probability
at least $\cutprob$.
This implies that $G$ contains an edge-cut of size at least $\cutsize n$,
where $n$ is the number of vertices of $G$, and has fractional cut covering
number at most $\cutcover$. The lower bound on the size of maximum edge-cut
also applies to random cubic graphs. Specifically, a random $n$-vertex cubic
graph a.a.s. contains an edge cut of size $\cutsize n$.
\end{abstract}

\section{Introduction}
\label{sect-cut-intro}
An~{\em edge-cut} in a~graph~$G=(V,E)$ defined by $X \subseteq V$ is the~set
of edges with exactly one end vertex in $X$ (and exactly one end vertex in $V \setminus X$).
A {\em maximum edge-cut} is an~edge-cut with the~maximum number of edges.
The size of a maximum edge-cut is an important graph parameter intensively
studied both in structural and algorithmic graph theory. From the algorithmic point of view, it
attracted a lot of attention because of an approximation algorithm based on the
semidefinite programming by Goemans and Williamson~\cite{bib-semidef} which achieves 
the~best possible approximation ratio under reasonable computational complexity
assumptions~\cite{bib-cut-inapprox}.
More specifically, assuming that the Unique Games Conjecture of Koth~\cite{bib-khot} holds, 
it is NP-hard to approximate the size of a maximum edge-cut in a graph $G$ within any 
factor greater than the approximation factor of the Goemans-Williamson algorithm.
On the other hand, there exists a polynomial-time algorithm for finding a 
maximum edge-cut in planar graphs~\cite{bib-hadlock}, and more generally in graphs
embeddable in a fixed orientable surface~\cite{bib-loebl}.
In this paper, we provide new structural results on maximum cuts in cubic graphs,
i.e., graphs with all vertices of degree three.

We prove a new lower bound on the size of a maximum edge-cut in a cubic graph
with no short cycle and in a random cubic graph.
Let us now mention earlier results. In 1990,
Z{\'y}ka~\cite{bib-zyka} proved that the~size of~the~maximum edge-cut in~cubic graphs
with large girth is at~least $9n/7 - o(n) = \cutzyka n - o(n)$.
A better bound $\cutindep n$ can be obtained from a recent result~\cite{bib-ind-girth}
on independent sets in cubic graphs with large girth.
The~asymptotic lower bound for a~maximum edge-cut in random cubic graphs of
$\cutwormald n$ was given by D{\' i}az, Do, Serna and
Wormald~\cite{bib-wormald-cut}. The experimental evidence suggests that almost all
$n$-vertex cubic graphs contain an edge-cut of size at least $\cutzdeblow n$~\cite{bib-zdeborova-private}.
On the~other hand, the~best known upper bound
is $\cutupper m = \cutuppern n$ which applies both to random cubic graphs and
cubic graphs with large girth. The~upper bound was announced by McKay~\cite{bib-mckay-cut},
its rigorous proof can be found in~\cite{bib-hh}. 
The problem could also be translated to a problem in statistical physics and applying non-rigorous
methods suggests that the size of a maximum edge-cut for almost all $n$-vertex graphs is at most
$\cutzdebup n$~\cite{bib-zdeborova}.

The problems of determining the size of a maximum edge-cut in random cubic graphs (more generally in random regular
graphs) and in cubic (regular) graphs with large girth are closely related.
On one hand, Wormald showed in~\cite{bib-wormald-shortcycles} that a~random
cubic graph asymptotically almost surely (a.a.s.) contains only $o(n)$ cycles
shorter than a~fixed integer $g$. Therefore, we can a.a.s. remove a~small
number (which means $o(n)$) of vertices to obtain a~subgraph with large girth
and only $o(n)$ vertices of degree less than three.

On the~other hand, Hoppen and Wormald~\cite{bib-hoppen-private}
have recently developed a~technique for translating many results for random
$r$-regular graphs to $r$-regular graphs with sufficiently large girth.
In particular, they are able to translate bounds obtained by analyzing
the~performance of so-called locally greedy algorithms for a~random regular 
graphs. These algorithms and their analysis provide the~currently best known
asymptotic bounds to many parameters of random regular graphs, for example
an~upper bound on~the~size of~the~smallest dominating set~\cite{bib-wormald-dom}.
The~main tool for the~analysis of such algorithms as well as for analysis of many other random
processes is the~{\em differential equation method} developed by~Wormald~\cite{bib-wormald}.

Bounds on maximum edge-cuts are closely related to the concept of fractional cut coverings.
A~{\em fractional cut covering} of a~graph $G$ is a parameter analogous to a~fractional coloring of $G$.
It was first introduced by {\v S}{\' a}mal~\cite{bib-samal} under the name cubical colorings;
he also related this parameter to graph homomorphisms.
These ideas are further developed in~\cite{bib-samal-endm, bib-samal-arxiv}.
The~aim is to assign non-negative weights to edge-cuts in $G$ in such a way that for 
each edge~$e$ of $G$ the~sum of weights of the~cuts containing $e$ is at least one. 
The~{\em fractional cut covering number} is the~minimum sum of weights of cuts 
forming a~fractional cut covering. Our approach in this paper gives also an upper bound
for the~fractional cut covering number of cubic graphs with sufficiently large girth.

\section{New results}
\label{sect-cut-thm}

The~main result of this paper is the~following.
\begin{theorem}
\label{thm-cut-dist}
There exists an absolute constant $g_0$ such that the following holds. If $G$
is a~cubic graph with girth at least $g_0$, then there exists a~probability
distribution on edge-cuts in $G$ such that each edge of $G$ is contained in
an~edge-cut drawn according to this distribution with probability at~least~$\cutprob$.
\end{theorem}
Proof of the Theorem~\ref{thm-cut-dist} actually provides that $g_0 \le \cutgirth$.

Before presenting the~proof of Theorem~\ref{thm-cut-dist}, let us state four corollaries of this theorem.
First, by considering the~expected size of an~edge-cut drawn according to the~distribution from Theorem~\ref{thm-cut-dist},
we get the~following.
\begin{corollary}
\label{cor-cut-size}
There exists an absolute constant $g_0$ such that 
every $n$-vertex cubic graph with girth at least $g_0$ contains an~edge-cut of size at least $\cutsize n$.
\end{corollary}

We can also translate Theorem~\ref{thm-cut-dist} to subcubic graphs with large girth.
\begin{corollary}
\label{cor-cut-subcubic}
There exists an absolute constant $g_0$ such that the following holds.
If $G$ is a graph with maximum degree at most three and girth at least $g_0$, then there exists a~probability
distribution on edge-cuts in $G$ such that each edge of $G$ is contained in
an~edge-cut drawn according to this distribution with probability at~least~$\cutprob$.
In particular, $G$ contains an~edge-cut of~size at~least~$\cutprob m$, where $m$ is the number of edges of $G$.
\end{corollary}
\begin{proof}
Fix $g_0$ to be the constant given by Theorem~\ref{thm-cut-dist}, and let $n_1$
and $n_2$ be the numbers of vertices of $G$ with degree one and two,
respectively. Clearly, we may assume that $G$ has no isolated vertices.
Let $R$ be a~$(2n_1 + n_2)$-regular graph with girth at least
$g_0$. There exists such a graph, since a random cubic graph has with positive probability
girth at least $g_0$ for every fixed value of $g_0$, which was proven by Bollob\'as~\cite{bib-bollob}
and independently by Wormald~\cite{bib-wormald-shortcycles}.
Replace each vertex of $R$ with a~copy of $G$ in such a~way that
the~edges of~$R$ are incident with vertices of degree one and two in the~copies
of~$G$ and the~resulting graph is cubic. Observe that the~obtained graph $H$
has girth at least $g_0$. 

Consider the probability distribution $\D$ given by Theorem~\ref{thm-cut-dist} on
edge-cuts in $H$ and fix an arbitrary copy $G'$ of the graph $G$ in $H$. For
every edge-cut $C \subseteq E(G)$ in $G$, we set the probability $p(C)$ to be
the probability that the edge-cut in $G'$ induced by a random edge-cut in $H$
drawn according to $\D$ is equal to $C$. This yields a probability distribution
on edge-cuts in $G$ with the required property.
%
\end{proof}

Since a~random cubic graph asymptotically almost surely contains only $o(n)$
cycles shorter than a~fixed integer $g$~\cite{bib-wormald-shortcycles},
the~lower bound on the~size of an~edge-cut also translates to random cubic graphs.
\begin{corollary}
\label{cor-cut-random}
A~random $n$-vertex cubic graph asymptotically almost surely contains an~edge-cut of size at least $\cutsize n - o(n)$.
\end{corollary}
\begin{proof}
Again, fix $g_0$ to be the constant given by Theorem~\ref{thm-cut-dist} and
let $G$ be a~randomly chosen $n$-vertex cubic graph.
The~results of~\cite{bib-wormald-shortcycles} imply then we can a.a.s. 
remove $o(n)$ vertices and obtain a~subgraph $G'$ with girth at least $g_0$.

Therefore, $G'$ has at least $1.5n - o(n)$ edges and by Corollary~\ref{cor-cut-subcubic},
there exists an edge-cut $C'$ in $G'$ of size at least $\cutsize n - o(n)$. Suppose that 
$X\subseteq V(G')$ is one of the sides of $C'$
and let $Y$ be the vertices removed from $G$. The edge-cut in $G$ with one side being $X \cup Y$
has size at least $\cutsize n - o(n)$.
\end{proof}

The~last corollary relates our results to the~problem of fractional coverings the~edges with edge-cuts.
We show how to construct from the~probability distribution given by Corollary~\ref{cor-cut-subcubic} a~fractional cut covering.
\begin{corollary}
\label{cor-cut-frac}
There exists an absolute constant $g_0$ such that 
every $n$-vertex graph $G$ with maximum degree at most three and girth at least
$g_0$ has the~fractional cut covering number at most $\cutcover$. 
\end{corollary}
\begin{proof}
Fix $g_0$ to be the constant given by Theorem~\ref{thm-cut-dist} and
consider the~probability distribution on edge-cuts in $G$ given by Corollary~\ref{cor-cut-subcubic}.
If the~probability of an~edge-cut $C$ to be drawn in this distribution is $p(C)$, assign $C$ weight $p(C)/ \cutprob$.
It is straightforward to verify that we have obtained a~fractional cut covering of weight $1 / \cutprob \le \cutcover$.
\end{proof}

\section{Structure of the~proof}
Our proof is inspired by~the~method which was developed by Lauer and Wormald in~\cite{bib-lauer}
for finding large independent sets in regular graphs with large girth.
This method was then extended by Hoppen~\cite{bib-hoppen}, who
improved the lower bound for independent sets and also proved a lower bound
for induced forests. The latter result can also be found in~\cite{bib-forest}.

In~order to~prove Theorem~\ref{thm-cut-dist}, we design a~randomized procedure for obtaining an~\mbox{edge-cut}
which resembles the procedure used in~\cite{bib-wormald-cut}. The~main difference between our procedure and the procedure from~\cite{bib-wormald-cut}
is that our procedure finds an edge-cut whose parts have slightly different sizes, while the procedure from~\cite{bib-wormald-cut}
finds an edge-cut whose parts have the same size. Surprisingly, at least at the first glance, this edge-cut constructed in an 
asymmetric way is larger than an edge-cut from~\cite{bib-wormald-cut}.

The~key tool for our analysis is the~independence lemma
(Lemma~\ref{lem-cut-indep}) which is given in~Section~\ref{sect-cut-indep}. This lemma is used to~simplify the~recurrence
relations appearing in the~analysis. The~recurrences describing the~behavior of~the~randomized procedure are derived in~Section~\ref{sect-cut-rec}.
The~actual performance of~the~procedure is based on setting up the~parameters of~the~procedure and solving the~recurrences numerically.
This is discussed in Section~\ref{sect-cut-param}.

The~sought probability distribution is obtained by processing a~cubic graph
$G=(V,E)$ by the~procedure which produces an edge-cut in it.
$G$ is processed in a~fixed number of rounds~$K$ and
the~required assumption on the~girth of $G$ will depend only on the~number~$K$.
We will iteratively construct two disjoint subsets $R \subseteq V$ and $B \subseteq V$;
the~vertices contained in $R$ are referred to as red vertices and those in $B$ as blue ones.
The~aim of the~procedure is to maximize the~number of red-blue edges.
The~vertices that are neither red nor blue will be called white.

All vertices are initially white.
In every round, each white vertex is recolored to red or blue with a~certain probability
depending on the~number of its red and blue neighbors, as well as on the~number of current round.
Once a~vertex is colored red or blue, its color stays the~same in all the~remaining rounds
of~the~procedure.

\section{Detailed description}
\label{sect-cut-model}
We now describe the~randomized procedure in more detail.
We first introduce some notation.
Let $I_j := \left\{ (r,b) : r \in \Nat, b \in \Nat, r+b\le j \right\}$,
i.e., the~set $I_j$ contains all pairs $r$ and $b$ of non-negative integers such that
$r+b\le j$. 
For example, $I_2 = \left\{ (0,0), (0,1), (1,0), (1,1), (2,0), (0,2) \right\}$.
Note that $|I_j|=\binom{j + 2}{2}$.
Let $G=(V,E)$ be a~cubic graph and $v$ a~vertex of $G$. Throughout the
analysis, $r(v)$ will refer to the~number of red neighbors of $v$ and $b(v)$
to the~number of its blue neighbors. Therefore, $3-r(v)-b(v)$
is the~number of the~white neighbors of $v$. If the~vertex $v$ is clear from the
context, we just use $r$ and $b$ instead of $r(v)$ and $b(v)$.

Our randomized procedure is parametrized by the~following parameters:
\begin{itemize}
\item an~integer $K$,
\item probabilities $P_k^{r,b}(W)$ for all $k \in [K]$ and $(r,b) \in I_3$ ,
\item probabilities $P_k^{r,b}(R)$ for all $k \in [K]$ and $(r,b) \in I_3$ and
\item probabilities $P_k^{r,b}(B)$ for all $k \in [K]$ and $(r,b) \in I_3$ .
\end{itemize}
We require that $P_k^{r,b}(W) + P_k^{r,b}(R) + P_k^{r,b}(B) = 1$
for all $k \in [K]$ and $(r,b) \in I_3$.
The precise values of these probabilities will be defined in Section~\ref{sect-cut-param}.
\smallskip

The~integer $K \in \Nat$ denotes the~number of rounds that are performed.
Throughout the~procedure, vertices of the~input graph $G$ have one of the~three colors: white~(W),
red~(R) and blue~(B).
Let $W_k \subseteq V(G)$ denote the~set of~white vertices after the~$k$-th round.
Analogously, we define $R_k$ and $B_k$ as the~sets of~red vertices and blue vertices, respectively. 
As we have already mentioned, at the~beginning of the~process
$W_0 := V, R_0 := \emptyset$ and $B_0 := \emptyset$.
For $(r,b) \in I_3$ we define $W_k^{r,b} \subseteq W_k$ to be the~set of white vertices with exactly $r$ red
neighbors and $b$ blue neighbors. Hence the~sets $W_k^{r,b}$ forms a~partition of~$W_k$ for every $k$.
Note that $W_0^{0,0}=V$ and $W_0^{r,b} = \emptyset$ for all $(r,b) \in I_3 \setminus \left\{ (0,0)\right\}$.

Consider the~coloring of $G$ obtained after the~$k$-th round. The~$(k+1)$-th round
of the~procedure is performed as follows.
Let $v$ be a~vertex from $W_k^{r,b}$. With probability $P_{k+1}^{r,b}(R)$ we
change the~color of $v$ to red, with probability $P_{k+1}^{r,b}(B)$ we recolor
it to blue, and with probability $P_{k+1}^{r,b}(W)$ it remains white.
If $v$ is after the~$k$-th round colored red or blue, it will not change its color
during the~$(k+1)$-th round.

Before we can proceed further, we have to introduce some additional notation.
For a~vertex $v \in V(G)$ let $T^d_v$ denote the~subgraph of $G$ induced by vertices
at the~distance from $v$ at most $d$. Observe that if the~girth of $G$ is larger than $2d+1$,
then the~subgraph $T^d_v$ is a~tree.

We show that if the~girth of $G$ is sufficiently large, then the~probabilities that after the~$k$-th round
a~vertex $v$ has white, red or blue color, respectively, do not depend on the~choice of $v$.
To do so, let is start with the~following proposition.
\begin{proposition}
\label{prop-cut-ternarylike}
Let $G$ be a~cubic graph 
and $v$ a~vertex of $G$.
For every $k\in [K]$ the~probability that the~subgraph $T^{K-k}_v$ has 
a~certain coloring after the~$k$-th round is determined by the~coloring of
$T^{K-k+1}_v$ after the~$(k-1)$-th round.
\end{proposition}
\begin{proof}
The~color of a~vertex $u \in T^{K-k}_v$ after the~$k$-th round depends only
on the~colors of $u$ and its neighbors after the~$(k-1)$-th round.
Since all the~neighbors of $u$ are contained in $T^{K-k+1}_v$, the~proposition follows.
\end{proof}
Suppose that the~girth of $G$ is at least $2K$. For any $k \in [K]$ the
structure of a~subgraph $T^{K-k}_v$ does not depend on the~choice of $v$, i.e.,
it is always a~tree with all inner vertices of degree three. Therefore, by
a~simple inductive argument on $k$ together with Proposition~\ref{prop-cut-ternarylike}, we conclude
that all the~following probabilities do not depend on the~choice of $v$:
\begin{equation*}
w_k := \Pr{v \in W_k} \; , \quad r_k := \Pr{v \in R_k} \; , \quad b_k := \Pr{v \in B_k} \tecka
\end{equation*}
Analogously, for any $k \in [K-1]$ and $(r,b) \in I_3$, the~probability that after the~$k$-th round a~vertex $v$ 
is white and has $r$ red neighbors and $b$ blue neighbors does not depend on the~choice of $v$ as well.
Therefore, we can define
\begin{equation*}
w^{r,b}_k := \Pr{v \in W^{r,b}_k \cond v \in W_k} \tecka
\end{equation*}

If the~girth of $G$ is at least $2K+1$, 
the~same reasoning as before yields the~following. 
The~probability that for an~edge $uv \in E(G)$ either $u$ is
red and $v$ is blue after the~$k$-th round, or $v$ is red and $u$ is blue after
the~$k$-th round does not depend on the~choice of $uv$. 
This probability will be denoted by
\begin{equation*}
p_k := \Pr{\left(u \in R_k \land v \in B_k\right) \lor \left(u \in B_k \land v \in R_k\right)} \tecka
\end{equation*}

\section{Independence lemma}
\label{sect-cut-indep}
In this section we present a~key tool we used in the~analysis
of the~randomized procedure. In general, our analysis follows the~approach used
in~\cite{bib-hoppen}.

If $G$ is a~cubic graph with girth at least
$2K+1$, $uv$ is an edge of $G$ and $d$ is an integer between $0$ and $K-1$,
$T_{v,u}^d$ denotes the~component of $T_v^d - u$ containing
the~vertex $v$. We refer to $v$ as to the~root of $T_{v,u}^d$. From the
assumption on the~girth it follows that all the~subgraphs $T_{v,u}^d$ are
isomorphic to the~same rooted binary tree $\T^d$ of depth $d$. 

Let $k \in [K]$. For a~set $V' \subseteq V(G)$ let $c_k(V')$ denote
the~coloring of vertices $V'$ after the~$k$-th round. 
The~set of all colorings of $\T^{K-k}$ such that the~root of the~tree is white is denoted by $\C_k$.
Observe that by the~girth assumption and Proposition~\ref{prop-cut-ternarylike},
for any $\gamma \in \C_k$ the~probability $\Pr{c_k\left(T_{v,u}^{K-k}\right) =
\gamma}$ does not depend on the~edge $uv$. 

We are ready to prove the~main lemma of this section.
\begin{lemma}[Independence lemma]
\label{lem-cut-indep}
Consider the~randomized procedure with parameters $K$ and
$P_i^{r,b}(C)$, where $i\in [K], (r,b)\in I_3$ and $C \in \{W,R,B\}$.
Let $G$ be a~cubic graph with girth at least $2K+1$, $uv$ an~edge of $G$,
$k$ an~integer smaller than $K$ and $\gamma_u$ and $\gamma_v$ two colorings from $\C_k$.
Conditioned by the~event $uv \subseteq W_k$, the~events
$c_k\left(T^{K-k}_{v,u}\right)=\gamma_v$ and
$c_k\left(T^{K-k}_{u,v}\right)=\gamma_u$ are independent.
In other words, the~probabilities 
\begin{equation}
\label{prob-cut-indep-first}
\Pr{c_k\left(T^{K-k}_{v,u}\right)=\gamma_v \cond uv \subseteq W_k}
\end{equation}
and
\begin{equation}
\label{prob-cut-indep-sec}
\Pr{c_k\left(T^{K-k}_{v,u}\right)=\gamma_v \cond v \in W_k \land c_k\left(T^{K-k}_{u,v}\right)=\gamma_u}
\end{equation}
are equal.
\end{lemma}
\begin{proof}
The~proof proceeds by induction on $k$. After the~first round each vertex has a~color~$C$
with probability $P_1^{0,0}(C)$ independently of the colors of the other vertices. Hence, the~claim holds for $k=1$.

Assume now that $k>1$. By the~definition of the~conditional probability and the~fact that the~event $uv \subseteq W_k$
immediately implies that the~event $uv \subseteq W_{k-1}$ occurs, (\ref{prob-cut-indep-first}) is equal to
\begin{equation}
\label{prob-cut-indep-first_ex}
\frac{ \Pr{c_k\left(T^{K-k}_{v,u}\right)=\gamma_v \land u \in W_k \cond uv \subseteq W_{k-1} } }
     { \Pr{uv \subseteq W_k \cond uv \subseteq W_{k-1}} }
\tecka
\end{equation}
Analogously, (\ref{prob-cut-indep-sec}) is equal to
\begin{equation}
\label{prob-cut-indep-sec_ex}
\frac{ \Pr{c_k\left(T^{K-k}_{v,u}\right)=\gamma_v \land c_k\left(T^{K-k}_{u,v}\right)=\gamma_u \cond uv \subseteq W_{k-1} } }
     { \Pr{v \in W_k \land c_k\left(T^{K-k}_{u,v}\right)=\gamma_u \cond uv \subseteq W_{k-1} } }
\tecka
\end{equation}
We now expand the~numerator of~(\ref{prob-cut-indep-first_ex}).
{\small \begin{align*}
&\sum_{\gamma'_u \in \C_{k-1}} \sum_{\gamma'_v \in \C_{k-1}}
\Pr{c_{k-1}\left(T^{K-k+1}_{u,v}\right) = \gamma'_u \cond uv \subseteq W_{k-1}} \\
&\times \Pr{c_{k-1}\left(T^{K-k+1}_{v,u}\right) = \gamma'_v \cond v \in W_{k-1} \land c_{k-1}\left(T^{K-k+1}_{u,v}\right) = \gamma'_u} \\
&\times \Pr{u \in W_k \cond c_{k-1}\left(T^{K-k+1}_{u,v}\right) = \gamma'_u \land c_{k-1}\left(T^{K-k+1}_{v,u}\right) = \gamma'_v} \\
&\times \Pr{c_k\left(T^{K-k}_{v,u}\right)=\gamma_v \cond c_{k-1}\left(T^{K-k+1}_{u,v}\right) = \gamma'_u \land c_{k-1}\left(T^{K-k+1}_{v,u}\right) = \gamma'_v 
  \land u \in W_k}
\tecka
\end{align*} }
By the~induction hypothesis, for any two colorings $\gamma'_u,\gamma'_v \in \C_{k-1}$ the~probabilities 
\begin{equation*}
\Pr{c_{k-1}\left(T^{K-k+1}_{v,u}\right) = \gamma'_v \cond v \in W_{k-1} \land c_{k-1}\left(T^{K-k+1}_{u,v}\right) = \gamma'_u}
\end{equation*}
and
\begin{equation*}
\Pr{c_{k-1}\left(T^{K-k+1}_{v,u}\right) = \gamma'_v \cond uv \subseteq W_{k-1}}
\end{equation*}
are equal.

Since the~new color of $u$ is determined only by the~colors of the~neighbors of $u$, it follows that the~probabilities
\begin{equation*}
\Pr{u \in W_k \cond c_{k-1}\left(T^{K-k+1}_{u,v}\right) = \gamma'_u \land c_{k-1}\left(T^{K-k+1}_{v,u}\right) = \gamma'_v}
\end{equation*}
and
\begin{equation*}
\Pr{u \in W_k \cond c_{k-1}\left(T^{K-k+1}_{u,v}\right) = \gamma'_u \land v \in W_{k-1}}
\end{equation*}
are also equal.

Analogously, for any vertex $w \in T^{K-k}_{v,u} \setminus \{v\}$ the~new color of $w$ does not depend on $\gamma'_u$ at all.
Applying the~same reasoning for $v$ yields that the~probabilities
\begin{equation*}
\Pr{c_k\left(T^{K-k}_{v,u}\right)=\gamma_v \cond c_{k-1}\left(T^{K-k+1}_{u,v}\right) = \gamma'_u \land c_{k-1}\left(T^{K-k+1}_{v,u}\right) = \gamma'_v 
  \land u \in W_k}
\end{equation*}
and
\begin{equation*}
\Pr{c_k\left(T^{K-k}_{v,u}\right)=\gamma_v \cond c_{k-1}\left(T^{K-k+1}_{v,u}\right) = \gamma'_v \land u \in W_{k-1}}
\end{equation*}
are equal as well.
Note that in the~last equality we have also used that the~random choices 
of new colors for two arbitrary vertices in the~$(k+1)$-th round are independent.

By changing the~order of summation, we conclude that the~numerator of~(\ref{prob-cut-indep-first_ex}) is equal to
\begin{align*}
\Bigg(\sum_{\gamma'_u \in \C_{k-1}} & \Pr{c_{k-1}\left(T^{K-k+1}_{u,v}\right) = \gamma'_u \cond uv \subseteq W_{k-1}} \\
\times & \; \Pr{u \in W_k \cond c_{k-1}\left(T^{K-k+1}_{u,v}\right) = \gamma'_u \land v \in W_{k-1}} \Bigg) \\
\times \Bigg(\sum_{\gamma'_v \in \C_{k-1}} & \Pr{c_{k-1}\left(T^{K-k+1}_{v,u}\right) = \gamma'_v \cond uv \subseteq W_{k-1}} \\
\times & \; \Pr{c_k\left(T^{K-k}_{v,u}\right)=\gamma_v \cond c_{k-1}\left(T^{K-k+1}_{v,u}\right) = \gamma'_v \land u \in W_{k-1}} \Bigg)
\tecka
\end{align*}

Along the~same lines, the~denominator of~(\ref{prob-cut-indep-first_ex}) is equal to
\begin{align*}
\Bigg(\sum_{\gamma'_u \in \C_{k-1}} & \Pr{c_{k-1}\left(T^{K-k+1}_{u,v}\right) = \gamma'_u \cond uv \subseteq W_{k-1}} \\
\times & \; \Pr{u \in W_k \cond c_{k-1}\left(T^{K-k+1}_{u,v}\right) = \gamma'_u \land v \in W_{k-1}} \Bigg) \\
\times \Bigg(\sum_{\gamma'_v \in \C_{k-1}} & \Pr{c_{k-1}\left(T^{K-k+1}_{v,u}\right) = \gamma'_v \cond uv \subseteq W_{k-1}} \\
\times & \; \Pr{v \in W_k \cond c_{k-1}\left(T^{K-k+1}_{v,u}\right) = \gamma'_v \land u \in W_{k-1}} \Bigg)
\tecka
\end{align*}

Canceling out the~sum over $\gamma'_u$ which is the~same in both numerator and denominator of~(\ref{prob-cut-indep-first_ex}),
we derive that~(\ref{prob-cut-indep-first}) is equal to
\begin{equation}
\label{prob-cut-indep-first_final}
\begin{aligned}
\Bigg(
\Sum_{\gamma'_v \in \C_{k-1}} & \Pr{c_{k-1}\left(T^{K-k+1}_{v,u}\right) = \gamma'_v \cond uv \subseteq W_{k-1}} \\
\times & \; \Pr{c_k\left(T^{K-k}_{v,u}\right)=\gamma_v \cond c_{k-1}\left(T^{K-k+1}_{v,u}\right) = \gamma'_v \land u \in W_{k-1}} \Bigg) \\
\times \Bigg( \Sum_{\gamma'_v \in \C_{k-1}} & \Pr{c_{k-1}\left(T^{K-k+1}_{v,u}\right) = \gamma'_v \cond uv \subseteq W_{k-1}} \\
\times & \; \Pr{v \in W_k \cond c_{k-1}\left(T^{K-k+1}_{v,u}\right) = \gamma'_v \land u \in W_{k-1}} \Bigg)^{-1}
\tecka
\end{aligned}
\end{equation}

We apply the~same trimming to the~numerator and denominator of~(\ref{prob-cut-indep-sec_ex}).
The~numerator is first expanded to
\begin{align*}
\Bigg(\sum_{\gamma'_u \in \C_{k-1}} & \Pr{c_{k-1}\left(T^{K-k+1}_{u,v}\right) = \gamma'_u \cond uv \subseteq W_{k-1}} \\
\times & \; \Pr{c_k\left(T^{K-k}_{u,v}\right)=\gamma_u \cond c_{k-1}\left(T^{K-k+1}_{u,v}\right) = \gamma'_u \land v \in W_{k-1}} \Bigg) \\
\times \Bigg(\sum_{\gamma'_v \in \C_{k-1}} & \Pr{c_{k-1}\left(T^{K-k+1}_{v,u}\right) = \gamma'_v \cond uv \subseteq W_{k-1}} \\
\times & \; \Pr{c_k\left(T^{K-k}_{v,u}\right)=\gamma_v \cond c_{k-1}\left(T^{K-k+1}_{v,u}\right) = \gamma'_v \land u \in W_{k-1}} \Bigg)
\end{align*}
and the~denominator is then expanded to
\begin{align*}
\Bigg(\sum_{\gamma'_u \in \C_{k-1}} & \Pr{c_{k-1}\left(T^{K-k+1}_{u,v}\right) = \gamma'_u \cond uv \subseteq W_{k-1}} \\
\times & \; \Pr{c_k\left(T^{K-k}_{u,v}\right)=\gamma_u \cond c_{k-1}\left(T^{K-k+1}_{u,v}\right) = \gamma'_u \land v \in W_{k-1}} \Bigg) \\
\times \Bigg(\sum_{\gamma'_v \in \C_{k-1}} & \Pr{c_{k-1}\left(T^{K-k+1}_{v,u}\right) = \gamma'_v \cond uv \subseteq W_{k-1}} \\
\times & \; \Pr{v \in W_k \cond c_{k-1}\left(T^{K-k+1}_{v,u}\right) = \gamma'_v \land u \in W_{k-1}} \Bigg)
\tecka
\end{align*}

By canceling out the~sum over $\gamma'_u$, we obtain (\ref{prob-cut-indep-first_final}). Therefore the~expressions $(\ref{prob-cut-indep-first})$ 
and $(\ref{prob-cut-indep-sec})$ are equal.
\end{proof}

\section{Recurrence relations}
\label{sect-cut-rec}
In this section we derive recurrence relations for the~probabilities describing
the~behavior of the~randomized procedure.

Fix parameters $K$ and $P_k^{r,b}(C)$ for $k\in [K], (r,b)\in I_3$ and
$C\in\{W,R,B\}$. We will inductively show that the~probabilities describing
the~state of the~procedure after the~\mbox{$(k+1)$-th} round can be computed
using only the~probabilities describing the~state after the~\mbox{$k$-th}
round. This yields the~recurrence relations for the~probabilities, which is
the~main goal of this section.

We start with determining the~probabilities after the~initialization round.
It is easy to see that the~probabilities $r_1, b_1, w_1, p_1$ and $w_1^{r,b}$ are
\begin{align*}
& r_1 = P_1^{0,0}(R)\carka \\
& b_1 = P_1^{0,0}(B)\carka \\
& w_1 = 1-r_1-b_1\carka \\
& p_1 = 2\cdot P_1^{0,0}(R)\cdot P_1^{0,0}(B) \quad \mbox{and} \\
& w_1^{r,k}=\binom{3}{r} \binom{3-r}{b} \cdot \left(P_1^{0,0}(R)\right)^r \cdot \left(P_1^{0,0}(B)\right)^b
\cdot \left(1-P_1^{0,0}(R)-P_1^{0,0}(B)\right)^{3-r-b} \\ & \mbox{for } (r,b) \in I_3 
\tecka
\end{align*}

Next, we show how to compute the~probabilities $r_{k+1}, b_{k+1}$ and $w_{k+1}$ from $r_k, b_k, w_k$ and $w_k^{r,b}$.
We start with the~formula for $r_{k+1}$. If a~vertex $v$ is colored red after the~\mbox{$(k+1)$-th} round,
then after the~\mbox{$k$-th} round, it was either already colored red, or it was white, had $r$ red neighbors,
$b$ blue neighbors and it was recolored to red. The~latter happened with probability $P_{k+1}^{r,b}(R)$. 
The~probability of the~first event is $r_k$ and that of the~second event is $w_k \cdot w_k^{r,b} \cdot P_{k+1}^{r,b}(R)$.
This yields that
\begin{equation*}
r_{k+1}=r_k + w_k \cdot \sum_{(r,b)\in I_3}{w_k^{r,b} \cdot P_{k+1}^{r,b}(R)}\tecka
\end{equation*}
Analogously, we can compute 
\begin{equation*}
b_{k+1}=b_k + w_k \cdot \sum_{(r,b)\in I_3}{w_k^{r,b} \cdot P_{k+1}^{r,b}(B)}\carka
\end{equation*}
and finally $w_{k+1}$ is given by 
\begin{equation*}
w_{k+1}=1-r_{k+1}-b_{k+1}\tecka
\end{equation*}

Before we proceed with the~recurrences for $p_{k+1}$ and $w_{k+1}^{r,b}$, let us introduce some auxiliary
notation. All of the~following quantities are fully determined by $w_k^{r,b}$, but this notation will help to make 
the~formulas simpler.
We start with probability that a~vertex $v$ has white color after the
\mbox{$(k+1)$-th} round conditioned by the~event it had white color after the~\mbox{$k$-th}
round. This quantity will be denoted by $w_{\to k+1}$. It is straightforward to check that
\begin{equation*}
w_{\to k+1} := \Pr{v \in W_{k+1} \cond v \in W_k} = \sum_{(r,b) \in I_3}{w_k^{r,b} \cdot P_{k+1}^{r,b}(W)}
\tecka
\end{equation*}
Next, we consider the~probability that the~vertex $u$ is white after the~{$k$-th} round conditioned
by the~event that a~fixed neighbor $v$ of $u$ is white after the~{$k$-th} round. This will be denoted by $q_k^{W-W}$.
We claim that 
\begin{equation*}
q_k^{W-W} := \Pr{uv \subseteq W_k \cond v \in W_k} = \sum_{(r,b) \in I_2}{ \frac{3-r-b}{3} \cdot w_k^{r,b} } \;
\tecka
\end{equation*}
First observe that the events $v \in W_k^{r,b}$, where $(r,b)\in I_3$, form a partition of the event $v\in W_k$,
and for $(r,b) \in I_3\setminus I_2$ the probability that $u$ is white after the~{$k$-th} round is equal to zero.
Suppose that $v\in W_k^{r,b}$, i.e., it has $r$ red neighbors, $b$ blue neighbors (and $3-r-b$ white neighbors) after the~{$k$-th} round.
This happens with probability $w_k^{r,b}$. Since $u$ is a~fixed neighbor of $v$, it has white color after the~{$k$-th} round 
with probability $(3-r-b)/3$.

Finally, for a~color $C \in \{W,R,B\}$ and an~edge $e=uv$,
$q_{\to k+1}^{(C)}$ denotes the~probability that $u$ has the~color $C$ after the~{$(k+1)$-th} round conditioned by the~event
that both $u$ and $v$ were white after the~{$k$-th} round.
We infer from the~definition of the~conditional probability that
{\small \begin{align*}
q_{\to k+1}^{(R)} &:= \Pr{u \in R_{k+1} \cond uv \subseteq W_k} = 
\Sum_{(r,b) \in I_2}\frac{{ w_k^{r,b} \cdot (3-r-b) \cdot P_{k+1}^{r,b}(R) } }{ 3 \cdot q_k^{W-W} } \; \carka\\
q_{\to k+1}^{(B)} &:= \Pr{u \in B_{k+1} \cond uv \subseteq W_k} = 
\Sum_{(r,b) \in I_2}\frac{{ w_k^{r,b} \cdot (3-r-b) \cdot P_{k+1}^{r,b}(B) } }{ 3 \cdot q_k^{W-W} } \; \carka\\
q_{\to k+1}^{(W)} &:= \Pr{u \in W_{k+1} \cond uv \subseteq W_k} = 
\Sum_{(r,b) \in I_2}\frac{{ w_k^{r,b} \cdot (3-r-b) \cdot P_{k+1}^{r,b}(W) } }{ 3 \cdot q_k^{W-W} } \; \tecka
\end{align*} }

We are now ready to present the~remaining recurrences. Let us start with $p_{k+1}$, i.e., the~probability 
than an~edge $e=uv$ is red-blue after the~$(k+1)$-th round.
Note that once we color a~vertex $x$ with either red or blue color, 
the~color of $x$ in the~future rounds will stay the~same. Therefore, we can split the~contribution to $p_{k+1}$
to the~following four types.
\begin{enumerate}
\item $e \cap W_k=\emptyset$ : This event happens with probability $p_k$ and the~colors stay the~same.

\item $e \cap W_k=\{v\}$ : Suppose first that $u$ is blue after the~{$k$-th} round. The~probability that we have such configuration after \mbox{$k$-th}
round is $w_k \cdot \sum_{(r,b) \in I_3} {w_k^{r,b} \cdot b/3}$. In this case, the~edge $e$ become red-blue after
the~{$(k+1)$-th} round with probability $P_{k+1}^{r,b}(R)$.
Analogously, if $u$ is red after the~{$k$-th} round, the~contribution of this case is 
$w_k \cdot \sum_{(r,b) \in I_3} { w_k^{r,b} \cdot P_{k+1}^{r,b}(B) \cdot r / 3 }\tecka$

\item $e \cap W_k=\{u\}$ : This case is symmetric to the~previous one.

\item $e \subseteq W_k$ : The~probability that $v$ has white color after the~{$k$-th} round is $w_k$. 
With probability $w_k^{r,b} \cdot (3-r-b)/3$, $v$ has $r$ red neighbors, $b$ blue neighbors, and $u$ is white after the~{$k$-th} round.
The~probability that $v$ becomes red after the~{$(k+1)$-th} round is $P_{k+1}^{r,b}(R)$, and using the~independence lemma (Lemma~\ref{lem-cut-indep}) the
neighborhood of $u$ does not depend on the~colors of the~other neighbors of $v$. Therefore, the~probability that $u$ becomes blue after the~{$(k+1)$-th} round
is $q_{\to k+1}^{(B)}$. On the~other hand, the~probability that after the~{$(k+1)$-th} round $v$ becomes red and $u$ becomes blue is
$P_{k+1}^{r,b}(B) \cdot q_{\to k+1}^{(R)}$.
\end{enumerate}
The~analysis just presented yields that
\begin{align*}
p_{k+1} = p_k
& + \frac{w_k}{3} \cdot \sum_{(r,b) \in I_3} w_k^{r,b} \cdot P_{k+1}^{r,b}(R) \cdot \left( 2b+ (3-r-b) \cdot q_{\to k+1}^{(B)} \right) \\
& + \frac{w_k}{3} \cdot \sum_{(r,b) \in I_3} w_k^{r,b} \cdot P_{k+1}^{r,b}(B) \cdot \left( 2r+ (3-r-b) \cdot q_{\to k+1}^{(R)} \right)
\tecka
\end{align*}

We finish this section with the~recurrence relations for the~probabilities $w_{k+1}^{r,b}$.
Observe that
\begin{equation}
\label{prob-cut-rec-whitedegrees}
w_{k+1}^{r,b} = \frac{ \Pr{v \in W_{k+1}^{r,b}} }{ \Pr{v \in W_{k+1}} } = 
\frac{ \Pr{v \in W_{k+1}^{r,b} \cond v \in W_k} }{ \Pr{v \in W_{k+1} \cond v \in W_k} }
\tecka
\end{equation}
The~second equality holds because each of the~events $v \in W_{k+1}$ and $v \in W_{k+1}^{r,b}$
immediately implies that the~event $v \in W_k$ occurs. The~denominator of~(\ref{prob-cut-rec-whitedegrees}) is equal
to $w_{\to k+1}$, so it remains to derive the~formula for the~numerator.

Let $N_k^W(v)$ denote the~set of white neighbors of $v$ after the~{$k$-th} round.
Using the~same argument as for deriving the~formula for $p_{k+1}$, the
color after the~{$(k+1)$-th} round of a~white neighbor $u \in N_k^W(v)$ will be red
with probability $q_{\to k+1}^{(R)}$. Analogously, it will be blue with probability 
$q_{\to k+1}^{(B)}$ and white with probability $q_{\to k+1}^{(W)}$. 
By Lemma~\ref{lem-cut-indep} and the~fact that in all rounds we recolor each
white vertex independently of the~others, the~new color of a~neighbor $u_1 \in N_k^W(v)$
does not depend on the~new color of another neighbor $u_2 \in N_k^W(v)$.
Now consider the probability that a vertex $v$ is white and has $r$ red and $b$ blue
neighbors after the~{$(k+1)$-th} round, i.e., $v \in W_{k+1}^{r,b}$, conditioned by the~event
$v \in W_k^{\bar{r},\bar{b}}$, where $\bar{r} \le r$ and $\bar{b} \le b$.
This probability is denoted by $w_{\to k+1}^{\bar{r},\bar{b} \to r,b}$.
We claim that $w_{\to k+1}^{\bar{r},\bar{b} \to r,b}$ is equal to
{\small
\begin{equation*}
 P_{k+1}^{\bar{r},\bar{b}}(W) \cdot
  \binom{3-\bar{r}-\bar{b}}{r-\bar{r}} \binom{3-r-\bar{b}}{b-\bar{b}} 
  \cdot \left(q_{\to k+1}^{(R)}\right)^{r-\bar{r}} 
  \cdot \left(q_{\to k+1}^{(B)}\right)^{b-\bar{b}}
  \cdot \left(q_{\to k+1}^{(W)}\right)^{3-r-b}
\tecka
\end{equation*}}
Indeed, $v$ stays white after the~{$(k+1)$-th} round with probability $P_{k+1}^{\bar{r},\bar{b}}(W)$.
Next, fix two disjoint subsets $Y$ and $Z$ of $N_k^W(v)$ of sizes $r-\bar{r}$ and $b-\bar{b}$, respectively.
This can be done in $\binom{3-\bar{r}-\bar{b}}{r-\bar{r}} \binom{3-r-\bar{b}}{b-\bar{b}}$ ways.
The probability that all vertices in $Y$ will be red after the~{$(k+1)$-th} round is equal to
$\left(q_{\to k+1}^{(R)}\right)^{r-\bar{r}}$. Analogously, all vertices in $Z$ will be
blue after the~{$(k+1)$-th} with probability $\left(q_{\to k+1}^{(B)}\right)^{b-\bar{b}}$, and the vertices
in $N_k^W(v) \setminus \left(Y \cup Z\right)$ will be white after the~{$(k+1)$-th} round with probability
$\left(q_{\to k+1}^{(W)}\right)^{3-r-b}$.
The above claim and the definition of the conditional probability imply that
{\begin{equation*}
w_{k+1}^{r,b} = 
\left(\sum_{ \substack{\bar{r}\le r \\ \bar{b} \le b} }
  { w_k^{\bar{r},\bar{b}} \cdot w_{\to k+1}^{\bar{r},\bar{b} \to r,b} } \right)
\Bigg/ w_{\to k+1}
\end{equation*} }
for every $(r,b) \in I_3$.

\section{Setting up the parameters}
\label{sect-cut-param}
In this section we set up the~parameters in~the~randomized procedure.
In the~first round, we pick a~vertex with a~small probability $p_0$ and color it
either red or blue. The~next rounds of the~procedure are split into two phases,
which consist of $K_1$ and $K_2$ rounds, respectively.
Therefore, the~total number of rounds $K$ is equal to $K_1 + K_2 + 1$.

In the~rounds of the~first phase, with probability $p_B$ ($p_R$), where $p_R \ll p_B$,
we color a~vertex with exactly one red (blue) neighbor by blue (red). If
a~vertex has at least two neighbors of the~same color, we color it with the~other
color with probability one. In all the~other cases we do nothing.

With one exception, the~rounds of the~second phase are performed identically to the~rounds of the~first phase.
The~exception is that a~white vertex with one red, one blue and one white neighbor is colored red
with probability $p_{RB}/2$ or blue with probability $p_{RB}/2$. 
The~choice of $p_{RB}$ is such that $p_{RB} \ll p_R$.

Specifically, we set:
\begin{itemize}
\item $K := K_1 + K_2 + 1$,

\item $P_1^{0,0}(R) := p_0/2 \carka \; P_1^{0,0}(B) := p_0/2 \carka$

\item $P_k^{r,b}(R) := 1$ for $(r,b) \in I_3 \cap \{(r,b): b \ge 2\}$ for $k \in [2,\dots, K] \carka$
\item $P_k^{r,b}(B) := 1$ for $(r,b) \in I_3 \cap \{(r,b): r \ge 2\}$ for $k \in [2,\dots, K] \carka$

\item $P_k^{0,1}(R) := p_R \carka \; P_k^{1,0}(B) := p_B$ for $k \in [2,\dots, K] \carka$

\item $P_k^{1,1}(R) := p_{RB}/2 \carka \; P_k^{1,1}(B) := p_{RB}/2$ for $k \in [K_1+2,\dots,K] \carka$

\item $P_k^{r,b}(R) := 0$ for all the~other choices of $r$ and $b$,
\item $P_k^{r,b}(B) := 0$ for all the~other choices of $r$ and $b$ and
\item $P_k^{r,b}(W) := 1 - P_k^{r,b}(R) - P_k^{r,b}(B)$ for $(r,b) \in I_3$.

\end{itemize}

The~recurrences presented in this chapter were solved numerically using a computer program.
The~particular choice of parameters used in the~program was 
$p_0=\cutprgPinit, p_B=\cutprgPblue, p_R=\cutprgPred, p_{RB}=\cutprgPrb$, \mbox{$K_1=\cutprgKone$}
and \mbox{$K_2=\cutprgKtwo$} (and hence $K=\cutprgK$).

The~choice of $K_1$ was made in such a~way that at the~end of the~first phase, i.e., after 
the~$(K_1 + 1)$-th round, the~probability that a~vertex is white and has exactly one non-white
neighbor is less than $\cutprgTHOLD$. 
Analogously, the~choice of $K_2$ was made in a~way that at the~end of the~process, i.e., after the~$K$-th round
the~probability that a~vertex is white is less than $\cutprgTHOLD$.

The code of the C program used for the computation can be downloaded from {\tt http://iuuk.mff.cuni.cz/\~{}volec/cubic-cut/}.
The output of the program with all the values of variables $p_k, r_k, b_k, w_k$ and $w^{r,b}_k$ for $k\in[K]$ and $(r,b)\in I_3$
computed for the given choice of parameters can also be found on the web page.
For the floating-point calculations, the program uses the MPFR library for a high-precision floating-point calculations with 
correct rounding~\cite{bib-float-mpfr}.
We used the {\em running error analysis} method (see, e.g., Section 2.5.1
from~\cite{bib-float-db}, or Section 3.3 from~\cite{bib-float-higham}) to upper
bound the rounding error coming from the~representation of floating-point
numbers. Setting the length of the mantissa of all the floating-point variables to $\cutprgprec$, we upper bound the rounding error
for all $p_K, r_K$ and $b_K$ by $\cutprgerrfla < \cutprgerr$.

Solving the~recurrences for the~above choice of parameters we have obtained that \mbox{$p_K > \cutprob$}.
The probability that a vertex $v$ is colored red at the end of the process, i.e. $r_K$, is equal to $\cutred$
and the probability that $b_K$ is equal to $\cutblue$.
In Figures~\ref{fig-p}--\ref{fig-w01-w11}, we plot the evolution of the
probabilities $p_k$, $r_k$, $b_k$, $w_k$, $w^{0,0}_k$, $w^{0,1}_k$ and
$w^{1,1}_k$. The vertical dashed line in each figure correspond to the end of the first phase.
The probabilities $w^{0,2}_k$, $w^{0,3}_k$, $w^{1,0}_k$, $w^{1,2}_k$
$w^{2,0}_k$, $w^{2,1}_k$ and $w^{3,0}_k$ are less than $10^{-3}$ for every $k
\in [K]$.

The above choice of the parameters is not the best possible. In particular, setting smaller values for
the parameters $p_{RB}, p_R$ and $p_0$ would produce a slightly larger edge-cut at the cost of stronger assumption
on the required girth. On the other hand, computer experiments on random cubic graphs of size $\randsize$
suggest that optimizing the parameters of this procedure cannot obtain significantly better upper bound than $\randprob$.

\begin{figure}
  \begin{center}
    \epsfxsize=100mm
    \epsfbox{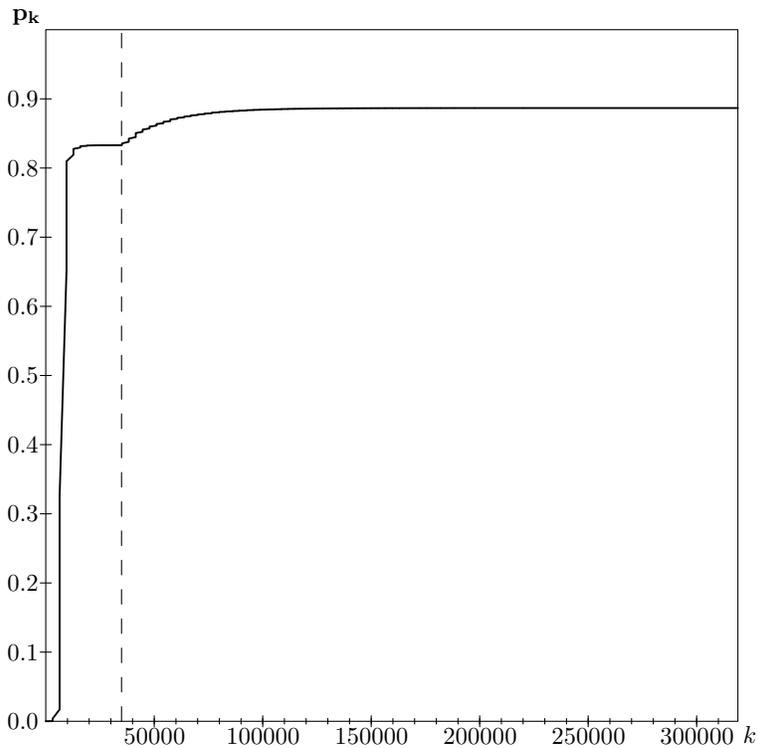}
  \end{center}
  \caption{Evolution of $p_k$}
  \label{fig-p}
\end{figure}

\begin{figure}
  \begin{center}
    \epsfxsize=135mm
    \epsfbox{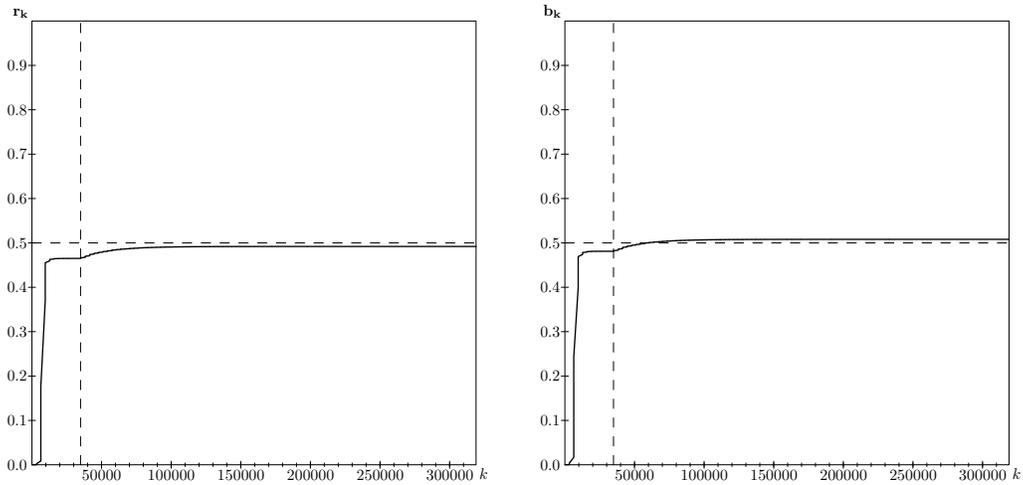}
  \end{center}
  \caption{Evolution of $r_k$ and $b_k$}
  \label{fig-r-b}
\end{figure}

\begin{figure}
  \begin{center}
    \epsfxsize=135mm
    \epsfbox{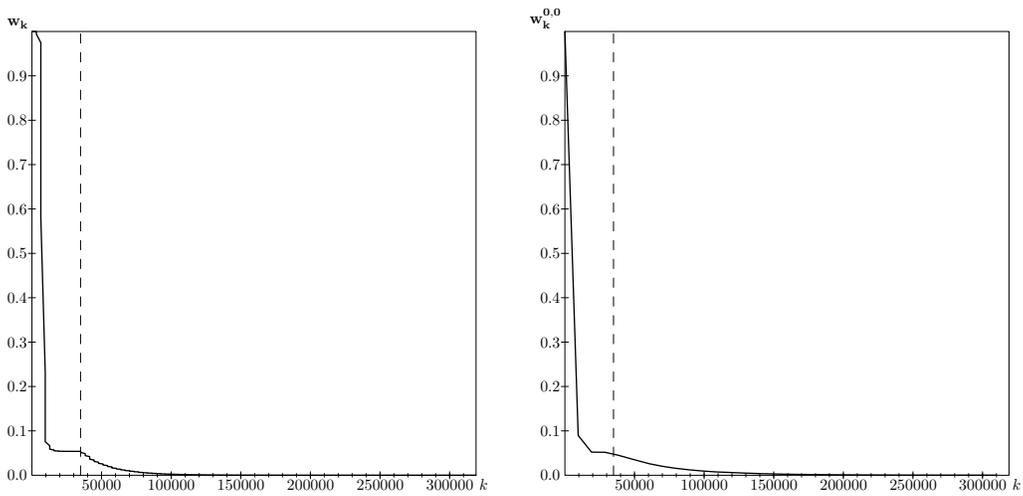}
  \end{center}
  \caption{Evolution of $w_k$ and $w^{0,0}_k$}
  \label{fig-w-w00}
\end{figure}

\begin{figure}
  \begin{center}
    \epsfxsize=135mm
    \epsfbox{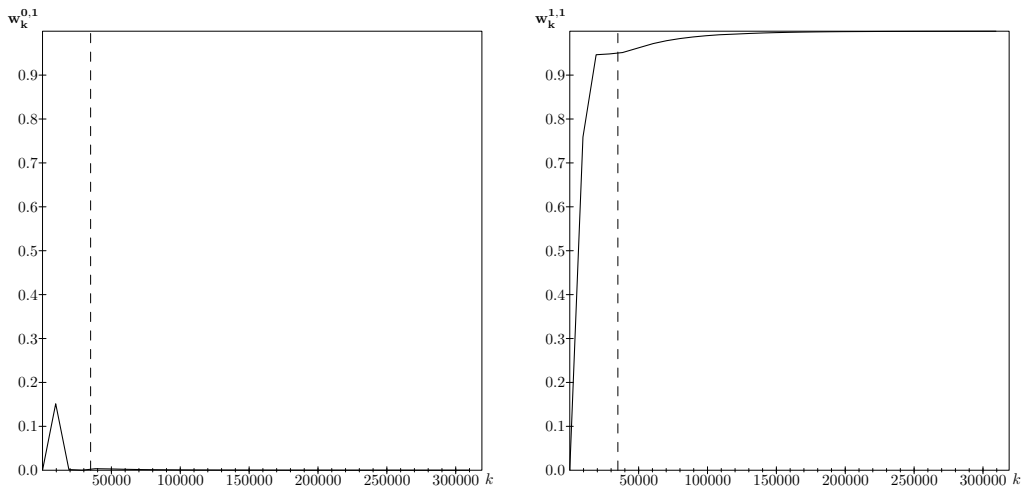}
  \end{center}
  \caption{Evolution of $w^{0,1}_k$ and $w^{1,1}_k$}
  \label{fig-w01-w11}
\end{figure}

The presented method is also applicable for $d$-regular graphs for $d\ge4$ analogously as 
the Hoppen's method~\cite{bib-hoppen} could be used for translating the results of
D{\' i}az, Do, Serna and Wormald~\cite{bib-wormald-cut4} and D{\' i}az, Serna and Wormald~\cite{bib-wormald-cutd}.

\paragraph{Acknowledgments}
The authors would like to thank Nick Wormald for pointing
out the reference~\cite{bib-lauer} and the anonymous referees
for careful reading the manuscript and their valuable
comments improving the presentation of the results.

\end{document}